\documentclass[11pt,a4paper]{amsart} 
\usepackage{amssymb,amscd,amsmath}%
\usepackage{mathrsfs,mathdots} 
\usepackage{mathtools}
\usepackage{float}
\usepackage[svgnames]{xcolor}
\usepackage{soul}
\usepackage{bm}
\usepackage{array,longtable}
\newcolumntype{C}{>{$}c<{$}}
\newcolumntype{R}{>{$}r<{$}}
\newcolumntype{L}{>{$}l<{$}}
\usepackage[unicode=true,pdfusetitle,
 bookmarks=true,bookmarksnumbered=false,bookmarksopen=false,
 breaklinks=false,pdfborder={0 0 0},pdfborderstyle={},backref=false,colorlinks=true]
 {hyperref}
\hypersetup{
 linkcolor={blue!60!black},citecolor={green!60!black},urlcolor={blue!80!black}}

\theoremstyle{plain}
\newtheorem{thm}{Theorem}[section]
\newtheorem{lem}[thm]{Lemma}
\newtheorem{pro}[thm]{Proposition}

\theoremstyle{remark}
\newtheorem{rem}[thm]{Remark}

\newtheorem{dfn}[thm]{Definition}
\newtheorem*{question*}{Question}

\newtheorem*{acknowledgements}{Acknowledgements}

\newcommand{\MSf}[1]{\footnote{}}

\newcommand{\TBf}[1]{\footnote{}}

\newcommand{\N}{\mathbb{N}}
\newcommand{\Z}{\mathbb{Z}}
\newcommand{\Q}{\mathbb{Q}}

\newcommand{\CC}{\mathbb{C}}

\newcommand{\R}{\mathbb{R}}
\newcommand{\tensor}{\otimes}

\newcommand{\bsL}{\boldsymbol{L}}
\newcommand{\bsf}{\boldsymbol{f}}
\newcommand{\bss}{\boldsymbol{s}}
\newcommand{\bsH}{\boldsymbol{H}}
\newcommand{\bsM}{\boldsymbol{M}}

\renewcommand{\epsilon}{\varepsilon}

\renewcommand{\phi}{\varphi}

\newcommand\restrict[1]{\raisebox{-.5ex}{\ensuremath{|_{#1}}}}
\newcommand{\unifmax}{r}

\DeclareMathOperator{\im}{im}

\makeatletter
\@namedef{subjclassname@2020}{
\textup{2020} Mathematics Subject Classification}
\makeatother

\author{Tomer Bauer} 

\author{Michael M.~Schein} \address{Department of Mathematics, Bar-Ilan University, Ramat Gan 5290002, Israel}


 
\begin{document}

 \title[Ideal growth in amalgamated powers]{Ideal growth in amalgamated powers of nilpotent rings of class two and zeta functions of quiver representations} 

\begin{abstract}
Let $L$ be a nilpotent algebra of class two over a compact discrete valuation ring $A$ of characteristic zero or of sufficiently large positive characteristic.  Let $q$ be the residue cardinality of $A$.  The ideal zeta function of $L$ is a Dirichlet series enumerating finite-index ideals of $L$.  We prove that there is a rational function in $q$, $q^m$, $q^{-s}$, and $q^{-ms}$ giving the ideal zeta function of the amalgamation of $m$ copies of $L$ over the derived subring, for every $m \geq 1$, up to an explicit factor.  More generally, we prove this for the zeta functions of nilpotent quiver representations of class two defined by Lee and Voll, and in particular for Dirichlet series counting graded submodules of a graded $A$-module.   If the algebra $L$, or the quiver representation, is defined over $\Z$, then we obtain a uniform rationality result.
\end{abstract}

 \maketitle

\section{Introduction}
\subsection{Amalgamated powers}
Let $A$ be a commutative unital ring.  The following definitions are tailored to the needs of this paper.  By an $A$-algebra we mean a free $A$-module $L$ of finite rank equipped with an $A$-bilinear multiplication $[ \ , \, ] \colon L \times L \to L$.  We consider nilpotent $A$-algebras of class at most two, namely those for which the inclusion $[L, L] \leq Z(L) = \{ x \in L : \forall y \in L, \, [x,y] = [y,x] = 0 \}$ holds.  Here $[L,L]$ is the $A$-submodule generated by the image of the multiplication map.  A (two-sided) $A$-ideal of $L$ is an $A$-submodule $I \leq L$ such that $[x,y] \in I$ and $[y,x] \in I$ for all $x \in I$ and $y \in L$.  Suppose that the number $a_n^{\triangleleft A} (L)$ of $A$-ideals of $L$ of index $n$ is finite for every $n \in \N$.  The ideal zeta function, introduced by Grunewald, Segal, and Smith in~\cite{GSS/88}, is the Dirichlet series
\[ \zeta_{L}^{\triangleleft A} (s) = \sum_{I \leq L} [L : I]^{-s} = \sum_{n = 1}^\infty a_n^{\triangleleft A} (L) n^{-s} ,\]
where the sum in the middle term runs over $A$-ideals of finite index and $s$ is a complex variable.  The sequence $a_n^{\triangleleft A} (L)$ grows polynomially in $n$ if the ring $A$ is either finitely generated (as a $\Z$-algebra) or semi-local~\cite[Theorem~1]{Segal/97}; in these cases, the zeta function $\zeta_{L}^{\triangleleft A} (s)$ converges on some right half-plane of $\mathbb{C}$.

We write $\N$ for the set $\{1, 2, \dots \}$ and define $\N_0 = \N \cup \{ 0 \}$.
Let $m \in \N$, and let $J \leq L$ be an ideal.  Define the {\emph{(direct) amalgamated $m$-th power}} $L^{\ast m}_J$ to be the amalgamated product of $m$ copies of $L$, with the copies of $J$ identified.  More precisely, given $x \in L$ and $1 \leq i \leq m$, set $e_i(x) \in L^m$ to be the $m$-tuple with $x$ in the $i$-th component and $0$ elsewhere.  Then define
\[ L^{\ast m}_J = L^m / \left\langle e_i(x) - e_j(x) : x \in J, 1\leq i, j \leq m \right\rangle.\]
We write $L^{\ast m}$ for $L^{\ast m}_{[L,L]}$.
A motivation of this work is to study the behavior of ideal zeta functions with respect to the operation of taking amalgamated powers.  

In this paper we focus on the case where $A$ is a discrete valuation ring that is compact or, equivalently, complete with finite residue field.  Such rings are precisely the valuation rings of local fields~\cite[\S II.4-5]{Serre/62}.  Let $q$ be the residue cardinality of $A$.
For the abelian $A$-algebra $A^n$ with $[x,y] = 0$ for all $x, y \in A^n$, it is known (cf.~\cite[Proposition~1.1]{GSS/88}) that
\begin{equation} \label{equ:abelian}
\zeta_{A^n}^{\triangleleft A} (s) = \prod_{i = 0}^{n - 1} (1 - q^{i - s})^{-1}.
\end{equation}
An $A$-Lie algebra is an $A$-algebra in which the multiplication is anti-symmetric and satisfies the Jacobi identity.  Consider now the Heisenberg $\Z$-Lie algebra $\mathcal{H} = \langle x, y, z \rangle_{\Z}$ with product determined by $[x,y] = z$ and $z \in Z(\mathcal{H})$.  In this case (cf.~\cite[Proposition~8.4]{GSS/88}):
\begin{equation} \label{equ:heisenberg}
 \zeta^{\triangleleft A}_{(\mathcal{H} \tensor_{\Z} A)^{\ast m}} (s) = \zeta^{\triangleleft A}_{A^{2m}} (s) \frac{1}{1 - q^{2m - (2m + 1)s}}.
 \end{equation}
 The ideal zeta functions $\zeta^{\triangleleft A}_{(\mathcal{H}^r \tensor_{\Z} A)^{\ast m}}(s)$, for $r \in \N$, were computed by the first author~\cite{Bauer/13}; while they are considerably more complicated than the right-hand side of~\eqref{equ:heisenberg}, the dependence on $m$ is of the same type.  See~\cite[\S 5.4]{CSV/19} for further variations.
It is thus natural to ask whether the form of the dependence of $ \zeta^{\triangleleft A}_{(\mathcal{H} \tensor_{\Z} A)^{\ast m}} (s)$ on $m$ exhibited in~\eqref{equ:heisenberg} is a general phenomenon.  We prove that it is.

\begin{thm} \label{thm:local}
Let $n, d \in \N$.  There exists $M_{n,d} \in \N$ such that the following holds:

Let $A$ be a compact discrete valuation ring either of characteristic zero or of characteristic at least $M_{n,d}$.  Let $q$ be the residue cardinality of $A$.  Let $L$ be a nilpotent $A$-algebra of class at most two such that $L / [L,L]$ is a free $A$-module of rank $n$ and $[L,L]$ is a free $A$-module of rank $d$.  There exists a rational function $W_L \in \Q(X, Y_0, Y_1, Y_2)$ such that
\[ \zeta^{\triangleleft A}_{L^{\ast m}} (s) = \zeta^{\triangleleft A}_{A^{mn}}(s) W_L (q, q^m, q^{-s}, q^{-ms})\]
for all $m \in \N$ and all $s \in \CC$ with $\mathrm{Re} \, s \gg_m 0$.  The function $W_L$ may be expressed with a denominator which is a product of factors of the form $(1 - X^{b} Y_0^{a_0} Y_1^{a_1} Y_2^{a_2})$ with $a_0, a_1, a_2 \in \N_0$ and $b \in \Z$.
\end{thm}

For any nilpotent $A$-algebra $L$, the ideal zeta function $\zeta^{\triangleleft A}_{L}(s)$ may be expressed as a rational function in $q^{-s}$ at least when $\mathrm{char} \, A = 0$; see~\cite[Theorem~3.5]{GSS/88} for the first general theorem of this form.  Theorem~\ref{thm:local} may be viewed as a refinement of this result for some nilpotent $A$-algebras of class at most two that makes precise the dependence of $\zeta^{\triangleleft A}_{L^{\ast m}}(s)$ on $m$; see Remark~\ref{rem:explicit.form} below.  The lower bound on $\mathrm{Re} \, s$ in the statement of Theorem~\ref{thm:local} and our other results depends on $m$; see the proof of Theorem~\ref{thm:quiver.main} below.

Since only one residue cardinality $q$ is considered in Theorem~\ref{thm:local}, the numerator of $W_L$ is not uniquely defined.  However,
if the algebra $L$ is defined over $\Z$, and hence may naturally be interpreted by extension of scalars as an algebra over any compact discrete valuation ring $A$, then we prove a version of Theorem~\ref{thm:local} that is uniform over all such $A$ of sufficiently large residue characteristic.

\begin{thm} \label{thm:main}
    Let $L$ be a nilpotent $\Z$-algebra of class at most two, and suppose that $L / [L,L]$ is a free $\Z$-module of rank $n$.  Then there exist $\unifmax, g, M \in \N$, rational functions $W_1, \dots, W_{\unifmax} \in \Q(X,Y_0, Y_1, Y_2)$, and $g$-ary formulae $\psi_{1}, \dots, \psi_{\unifmax}$ in the language of rings, such that the following holds for all compact discrete valuation rings $A$ of residue characteristic at least $M$ and all $m \in \N$:
For all $s \in \CC$ with $\mathrm{Re} \, s \gg_m 0$ we have
\begin{equation*} 
\zeta^{\triangleleft A}_{(L \tensor_{\Z} A)^{\ast m}}(s) = {\zeta^{\triangleleft A}_{A^{mn}}(s)} \sum_{i = 1}^{\unifmax} m_i (k) W_i(q, q^m, q^{-s}, q^{-ms}) ,
\end{equation*}
where $k$ is the residue field of $A$, with $q = |k|$ and 
$m_i (k) = | \{ \xi \in k^{g} : k \models \psi_i(\xi) \} |$.  Moreover, the rational functions $W_1, \dots, W_r$ may be expressed with denominators which are products of factors of the form $(1 - X^{b} Y_0^{a_0} Y_1^{a_1} Y_2^{a_2})$ with $a_0, a_1, a_2 \in \N_0$ and $b \in \Z$.
\end{thm}

\begin{rem}
If $L = \mathcal{H}$, then~\eqref{equ:heisenberg} fits into the framework of Theorem~\ref{thm:main} as follows: 
\[\zeta^{\triangleleft A}_{(\mathcal{H} \tensor_{\Z} A)^{\ast m}} (s) = \zeta^{\triangleleft A}_{A^{2m}} (s) \, m(k) \, W(q, q^m, q^{-s}, q^{-ms}),\] 
where $m(k) = | \{ \xi \in k: k \models \psi(\xi) \} |$ for $\psi(\eta) = \{ \eta = 0 \}$ and $W(X,Y_0, Y_1, Y_2) = \frac{1}{1 - Y_0^2 Y_1^{} Y_2^{2}}$.
\end{rem}

\begin{rem} \label{rem:explicit.form}
A consequence of Theorems~\ref{thm:local} and~\ref{thm:main} is that taking amalgamated powers does not increase the complexity of ideal zeta functions, apart from the explicit factor $\zeta^{\triangleleft A}_{A^{mn}}(s)$.  More precisly, there exists an expression for $\zeta^{\triangleleft A}_{L^{\ast m}}(s) / \zeta^{\triangleleft A}_{A^{mn}}(s)$ having uniform complexity for all $m$; we cannot exclude the possibility that it admits further cancellation for particular values of $m$.
Indeed, let $L$ be a nilpotent $\Z$-algebra of class at most two as in Theorem~\ref{thm:main}.  Then, for all compact discrete valuation rings $A$ of sufficiently large residue characteristic, the ideal zeta function 
$\zeta^{\triangleleft A}_{L^{\ast m}}(s)$ has the form
\[ \zeta^{\triangleleft A}_{L^{\ast m}}(s) = \zeta^{\triangleleft A}_{A^{mn}}(s)
\frac{\sum_{i = 1}^M \kappa_i q^{a_{i 0}m + a_{i1} - (b_{i0} m + b_{i1})s}}{\prod_{i = 1}^N (1 - q^{c_{i 0}m + c_{i1} - (d_{i0} m + d_{i1})s})}\]
for suitable coefficients $\kappa_i \in \Q$ and $a_{ij}, b_{ij}, c_{ij}, d_{ij} \in \Z$ that are independent of $m$.  Thus, apart from the explicit factor $\zeta^{\triangleleft A}_{A^{mn}}(s)$, the only dependence of $\zeta^{\triangleleft A}_{L^{\ast m}}(s)$ on $m$ is that the exponents of monomials appearing in the rational function are linear in $m$.  

The hypothesis of nilpotency class at most two in our results appears to be necessary.
The ideal zeta functions $\zeta^{\triangleleft \Z_p}_{M_3 \tensor_\Z \Z_p}(s)$ and $\zeta^{\triangleleft \Z_p}_{(M_3 \tensor_\Z \Z_p)^{*2}}(s)$ are computed in Theorems~2.26 and~2.30 of~\cite{duSWoodward/08}; here $M_3$ is a certain nilpotent $\Z$-Lie algebra of class {\emph{three}} and $M_3^{\ast 2}$ is the amalgamation of two copies of $M_3$ over the center.  The latter function is significantly more complex than the former, suggesting that no analogue of Theorem~\ref{thm:main} applies.
\end{rem}

The generality of Theorems~\ref{thm:local} and~\ref{thm:main} stands in stark contrast to the difficulty of computing ideal zeta functions explicitly.  No explicit ideal zeta functions are known for nilpotent $A$-algebras of class five or greater.  Even in class two, the known examples either have small rank, such as those listed in~\cite[\S8]{GSS/88} and~\cite[\S2]{duSWoodward/08}, or have a particular combinatorial structure rendering them tractable, such as those considered in~\cite{duS-ecI/01, ScheinVoll1/15, ScheinVoll2/16, Voll/19, CSV/19}.  The rational functions describing ideal zeta functions of base extensions to compact discrete valuation rings of nilpotent $\Z$-algebras of class two generically satisfy functional equations upon inversion of the variables~\cite[Theorem~C]{Voll/10}.

\subsection{Counting normal subgroups}
Theorem~\ref{thm:main} may be interpreted in the context of counting normal subgroups of finitely generated nilpotent groups.  Indeed, let $G$ be a finitely generated torsion-free nilpotent group.  Grunewald, Segal, and Smith~\cite{GSS/88} introduced the function
\begin{equation} \label{equ:normal.subgroup.zeta}
\zeta^\triangleleft_{G}(s) = \sum_{H \unlhd G} [G:H]^{-s},
\end{equation}
where the sum runs over all normal subgroups of $G$ of finite index.  There is an Euler decomposition $\zeta^\triangleleft_G (s) = \prod_p \zeta^\triangleleft_{G,p} (s)$, where the product runs over all primes and $\zeta^\triangleleft_{G,p} (s)$ is defined as in~\eqref{equ:normal.subgroup.zeta} but counts only the normal subgroups of $p$-power index~\cite[Proposition~1.3]{GSS/88}.
If $G$ is nilpotent of class two, then $L(G) = G/Z(G) \times Z(G)$ with the multiplication $[(g_{1}Z(G),z_{1}), (g_{2}Z(G),z_{2})] = (Z(G), g_1 g_2 g_1^{-1} g_2^{-1})$ is a nilpotent $\Z$-Lie algebra of class two, and there is a one-to-one correspondence between normal subgroups of $G$ and ideals of $L(G)$ that preserves inclusion and index; see~\cite[Proposition~3.1]{BKO/22} for further detail.  Thus $\zeta^\triangleleft_G (s) = \zeta^{\triangleleft \Z}_{L(G)} (s)$ and $\zeta^\triangleleft_{G,p} (s) = \zeta^{\triangleleft \Z_p}_{L(G) \tensor_{\Z} \Z_p}(s)$ for all primes $p$.  Any class-two nilpotent $\Z$-Lie algebra $L$ is isomorphic to $L(G)$ for a suitable group $G$, and it is easy to see that the operation $L^{\ast m}$ corresponds to amalgamating $m$ copies of $G$ over the derived subgroup $[G,G]$.  

\subsection{Comparison with related zeta functions} \label{sec:comparison}
It is natural to ask how zeta functions behave under operations on the underlying objects such as direct product, amalgamation, and base extension.  In this light we observe that
the ``preservation of complexity'' under amalgamated powers in Remark~\ref{rem:explicit.form} contrasts with behavior under base extension: given an $A$-Lie algebra $L$ and an extension $B/A$ of rings such that $B$ is free of finite rank as an $A$-module, we may consider $L \tensor_A B$ as an $A$-Lie algebra via restriction of scalars from $B$ to $A$ and seek to study $\zeta^{\triangleleft A}_{L \tensor_A B} (s)$.  For instance, if $B = A^r$, then $L \tensor_A B = L^r$ is the direct product of $r$ copies of $L$.  The complexity of $\zeta^{\triangleleft A}_{L \tensor_A B} (s)$ appears to grow rapidly with $\mathrm{rk}_A B$.  For the class of Lie algebras considered in~\cite[Theorem~4.21]{CSV/19}, for instance, it is possible to express $\zeta^{\triangleleft A}_{L \tensor_A B} (s)$ as a sum of explicit combinatorially defined functions, where the number of summands grows super-exponentially in $\mathrm{rk}_A B$.

Nearly the opposite behavior occurs if we consider the pro-isomorphic zeta functions
\[ \zeta^{\wedge A}_{L} (s) = \sum_{M \simeq L} [L : M]^{-s},\]
where the sum runs over all $A$-Lie subalgebras $M \leq L$ that are isomorphic to $L$; these are necessarily of finite index.  Let $L$ be a nilpotent $\Z$-Lie algebra, of arbitrary class, such that $Z(L) \leq [L,L]$.  By~\cite[Theorem~1.3(2)]{BGS/22}
there exist formulae $\widehat{\psi}_1, \dots, \widehat{\psi}_r$ and rational functions $\widehat{W}_1, \dots, \widehat{W}_r \in \Q(X,Y_1, Y_2)$ such that, for any number field $K$ of degree $m$ and any prime $p$ such that 
$L \tensor_{\Z} \Z_p$ is rigid in the sense of~\cite[Definition~3.8]{BGS/22}, we have
\begin{equation} \label{equ:pro-isomorphic.rational}
\zeta^{\wedge \Z_p}_{(L \tensor_\Z \Z_p) \tensor_\Z \mathcal{O}_K}(s) = \prod_{\mathfrak{p} | p} \left( \sum_{i = 1}^r \widehat{m}_i (k_{\mathfrak{p}}) \widehat{W}_i(q_{\mathfrak{p}}, q_{\mathfrak{p}}^m, q_{\mathfrak{p}}^{-s}) \right),
\end{equation}
where the product runs over places $\mathfrak{p}$ of $K$ dividing $p$, and $q_{\mathfrak{p}}$ is the cardinality of the residue field $k_{\mathfrak{p}} = \mathcal{O}_K / \mathfrak{p}$.  Here $\widehat{m}_i(k_{\mathfrak{p}}) = | \{ \xi \in k_{\mathfrak{p}}^e : k_{\mathfrak{p}} \models \widehat{\psi}_i(\xi) \} |$.  Interpreting~\eqref{equ:pro-isomorphic.rational} as a rational function in $q_{\mathfrak{p}}$ and $q_{\mathfrak{p}}^{-s}$, this means that the exponents of $q_{\mathfrak{p}}$ depend linearly on $m$, while all other coefficients and exponents are independent of $m$.  

Although~\cite[Theorem~1.3(2)]{BGS/22} involves a fine Euler decomposition, whereas Theorem~\ref{thm:main} includes the factor $\zeta^{\triangleleft A}_{A^{mn}}(s)$, the two theorems and their proofs share a similar spirit.  In both cases, the relevant zeta functions are interpreted as $p$-adic integrals whose domains of integration are definable sets and integrands are powers (possibly depending on $m$) of valuations of definable functions that do not depend on $m$.  The theorems thus become rationality and uniformity results for such integrals and are proved using arguments from the theory of $p$-adic integration.

By contrast, pro-isomorphic zeta functions do not behave well under amalgamated products.  For instance, the functions $ \zeta^{\wedge A}_{(\mathcal{H} \tensor_{\Z} A)^{\ast m}}(s)$ are computed in~\cite[Example~1.4(2)]{duSLubotzky/96} for $m \in \{ 2, 3 \}$ and~\cite[Theorem~5.10]{BGS/22} in general; their complexity appears to grow factorially with $m$.

\subsection{Quiver representations}
Our main results hold in greater generality than presented above.  
Let $Q = (V,E)$ be a finite quiver, namely a finite directed graph, possibly containing loops and multiple edges.  Here $V$ and $E$ are the sets of vertices and edges of $Q$, respectively.  Given an edge $e \in E$, we denote its tail by $t(e)$ and its head by $h(e)$.  
Set $[n] = \{1, 2, \dots, n \}$ for all $n \in \N$.

\begin{dfn} \label{def:quiver.rep}
Let $A$ be a commutative unital ring, and let $Q$ be a quiver.
\begin{enumerate}
\item
An \emph{$A$-representation}
of $Q$ is a pair $(\boldsymbol{L},\boldsymbol{f})$, where
$\bsL = \{ L_v \}_{v \in V}$ is a family of $A$-modules and $\bsf = \{f_e \colon L_{t(e)} \to L_{h(e)} \}_{e \in E}$ is a family of $A$-module homomorphisms.  We occasionally write $\boldsymbol{L}$ for $(\boldsymbol{L},\boldsymbol{f})$, especially in subscripts.
\item
An \emph{$A$-subrepresentation} of $(\boldsymbol{L}, \bsf)$ is a family $\boldsymbol{\Lambda} = \{ \Lambda_v \leq L_v \}_{v \in V}$ of $A$-submodules such that $f_e (\Lambda_{t(e)}) \leq \Lambda_{h(e)}$ for every $e \in E$.  Note that $(\boldsymbol{\Lambda}, \{ (f_e)\restrict{\Lambda_{t(e)}} \}_{e \in E})$ is an $A$-representation of $Q$.  We say that $\boldsymbol{\Lambda}$ has finite index in $\bsL$ if $[L_v : \Lambda_v] < \infty$ for all $v \in V$.
\end{enumerate}
\end{dfn}

Quiver representations, especially over fields, have attracted considerable interest over the last several decades; see, for instance, the expository works~\cite{Brion/12, DerksenWeyman/17}.
Zeta functions of nilpotent quiver representations over rings were introduced recently by Lee and Voll~\cite{LeeVoll/21}:

\begin{dfn}
Let $A$ be a commutative unital ring, and let $(\bsL, \bsf)$ be an $A$-representation of a quiver $Q = (V,E)$.  Assume that for every $(\kappa_v)_v \in \N^V$ there are only finitely many $A$-subrepresentations $\boldsymbol{\Lambda}$ of $\bsL$ such that $[L_v : \Lambda_v] = \kappa_v$ for all $v \in V$.  The zeta function of $(\bsL, \bsf)$ is the Dirichlet series 
$$ \zeta_{\boldsymbol{L}} (\boldsymbol{s}) = \sum_{\boldsymbol{\Lambda} \leq \boldsymbol{L}} \prod_{v \in V} [ L_v : \Lambda_v ]^{- s_v},$$
where the sum runs over all finite-index $A$-subrepresentations of $(\boldsymbol{L}, \bsf)$, there is a complex variable $s_v$ for every $v \in V$, and $\boldsymbol{s} = (s_v)_{v \in V}$.
\end{dfn}

To avoid repeating a long string of adjectives throughout the paper, we make the following ad hoc definitions.  The representations here called admissible are, in the terminology of~\cite{LeeVoll/21}, the homogeneous nilpotent $A$-representations of class at most two that are free of finite rank.
\begin{dfn} \label{def:admissible.rep}
Let $Q = (V,E)$ be a quiver, and let $A \neq 0$ be a commutative unital ring.
\begin{enumerate}
\item
An {\emph{admissible}} $A$-representation of $Q$ is an $A$-representation $(\bsL, \bsf)$ together with a decomposition $L_v = L_{v,1} \oplus L_{v,2}$ for every $v \in V$, where $L_{v,1}$ and $L_{v,2}$ are free $A$-modules of finite rank, such that $0 \oplus L_{t(e),2} \leq \ker f_e$ and $\im f_e \leq 0 \oplus L_{h(e),2}$ for every $e \in E$.
\item
Given an admissible $A$-representation $(\boldsymbol{L}, \bsf)$, set $n(v,i) = \mathrm{rk}_{A} L_{v,i}$ for $v \in V$ and $i \in [2]$.  The \emph{rank vector} of $(\boldsymbol{L}, \bsf)$ is $\boldsymbol{n}(\boldsymbol{L}, \bsf) = (n(v,1), n(v,2))_{v \in V} \in (\N_0^2)^V$.
\end{enumerate}
\end{dfn}

\begin{rem} \label{exm:lattices}
Let $L$ be a nilpotent $A$-algebra of class at most two, and let $J \leq L$ be an ideal such that $[L,L] \leq J \leq Z(L)$ and that $J$ and $L/J$ are free $A$-modules of finite rank.  Let $(b_1, \dots, b_{n+d})$ be an $A$-basis of $L$ such that $J = \langle b_{n+1}, \dots, b_{n+d} \rangle_A$.  
Set $L_1 = \langle b_1, \dots, b_n \rangle_A$ and $L_2 = J$.  Consider the quiver consisting of a single vertex $v_0$ and $2n$ loops.  It has an admissible $A$-representation $(\boldsymbol{L}, \boldsymbol{f})$, where $L_{v_0} = L = L_1 \oplus L_2$ (the equality is one of $A$-modules), and the $2n$ maps $f_e \colon L \to L$ are the left and right multiplications by $b_1, \dots, b_n$.  It is clear that $\zeta^{\triangleleft A}_{L} (s) = \zeta_{\boldsymbol{L}}(s)$. 
\end{rem}

Further important examples of quiver representation zeta functions are Dirichlet series enumerating graded ideals of a graded $A$-algebra and submodules of an $A$-algebra that are invariant under the action of a collection of endomorphisms~\cite{Rossmann/15, Rossmann/17, Voll/17}.  Indeed, quiver representation zeta functions are equivalent to the class of Dirichlet series counting graded submodules invariant under a collection of endomorphisms~\cite[\S1.3.3]{LeeVoll/21}.  

Let $Q = (V,E)$ be a quiver and let $m \in \N$.  Set $Q^{\ast m} = (V, E \times [m])$, where $t(e,i) = t(e)$ and $h(e,i) = h(e)$ for every $i \in [m]$.  Informally, the vertex set of $Q^{\ast m}$ is the same as that of $Q$, but every edge of $Q$ is replaced by $m$ edges with the same tail and head.

\begin{dfn} \label{def:quiver.product}
Let $(\boldsymbol{L},\bsf)$ be an admissible $A$-representation of $Q$ and let $m \in \N$.  The amalgamated $m$-th power of $(\boldsymbol{L},\bsf)$ is the following admissible $A$-representation $(\boldsymbol{L}^{\ast m}, \bsf^{\ast m})$ of $Q^{\ast m}$.  For every $v \in V$, set $L^{\ast m}_v = (L_{v,1})^m \oplus L_{v,2}$ and put $\bsL^{\ast m} = \{ L_v^{\ast m} \}_{v \in V}$.  For every $e \in E$ and $i \in [m]$, define the map $f_{e,i} \colon L^{\ast m}_{t(e)} \to L^{\ast m}_{h(e)}$ to be $f_e$ on the $i$-th component of $L_{t(e),1}^m$ and zero on the remaining components of $L^{\ast m}_{t(e)}$.  Then $\bsf^{\ast m}$ is the family $\{ f_{e,i} \}_{(e,i) \in E \times [m]}$.
\end{dfn}

\begin{rem} \label{exm:lattices.amalgamated}
The previous definition is compatible with the notion of amalgamated powers of $A$-algebras via the dictionary of Remark~\ref{exm:lattices}.  Indeed, if $(\bsL, \bsf)$ is the quiver representation defined there, then $\zeta_{(\bsL^{\ast m}, \bsf^{\ast m})} = \zeta^{\triangleleft A}_{L^{\ast m}_J}$.
\end{rem}

We can now state our main results in terms of zeta functions of quiver representations.

\begin{thm} \label{thm:quiver.local}
Let $Q = (V,E)$ be a quiver, and let $\boldsymbol{n} \in (\N_0^2)^V$.  There exists $M_{\boldsymbol{n}} \in \N$ such that the following holds:

Let $A$ be a compact discrete valuation ring either of characteristic zero or of characteristic at least $M_{\boldsymbol{n}}$, and let $q$ be its residue cardinality.  Let $(\boldsymbol{L},\bsf)$ be any admissible $A$-representation of $Q$ with rank vector $\boldsymbol{n}(\boldsymbol{L},\bsf) = \boldsymbol{n}$.  There exists a rational function $W_{\boldsymbol{L}} \in \Q(X, Y_0, \boldsymbol{Y}_1, \boldsymbol{Y}_2)$, where $\boldsymbol{Y}_1 = (Y_{1,v})_{v \in V}$ and $\boldsymbol{Y}_2 = (Y_{2,v})_{v \in V}$, such that
$$ \zeta_{\boldsymbol{L}^{\ast m}}(\boldsymbol{s}) = \left( \prod_{v \in V} \zeta^{\triangleleft A}_{A^{m n(v,1)}}(s_v) \right) W_{\boldsymbol{L}}(q, q^m, (q^{-s_v}, q^{- m s_v})_{v \in V})$$
for all $m \in \N$ and all $\boldsymbol{s} \in \CC^V$ with $\mathrm{Re} \, s_v \gg_m 0$ for all $v \in V$.  The rational function $W_{\bsL}$ may be expressed over a denominator which is a product of factors having the form $(1 - X^{b} Y_0^{a_0} \prod_{v \in V} (Y_{1,v}^{a_{1,v}} Y_{2,v}^{a_{2,v}}))$ with $a_0, a_{1,v}, a_{2,v} \in \N_0$ and $b \in \Z$.
\end{thm}

We emphasize that the lower bound $M_{\boldsymbol{n}}$ in Theorem~\ref{thm:quiver.local} depends only on the rank vector $\boldsymbol{n}$, not on the representation $(\bsL, \bsf)$ or on $m$.

Given a $\Z$-representation $\boldsymbol{L}$ of $Q$, we may naturally extend scalars to any ring $A$ to obtain an $A$-representation $\boldsymbol{L} \tensor_{\Z} A$ of $Q$; see, for instance,~\cite[\S1.1]{LeeVoll/21}.  
Zeta functions of base extensions to compact discrete valuation rings of admissible quiver $\Z$-representations generically satisfy functional equations~\cite[Theorem~1.7]{LeeVoll/21}, generalizing the result for ideal zeta functions mentioned above.
By Remark~\ref{exm:lattices.amalgamated} the following claim generalizes Theorem~\ref{thm:main}, just as Theorem~\ref{thm:quiver.local} generalizes Theorem~\ref{thm:local}.

\begin{thm} \label{thm:quiver.main}
Let $(\boldsymbol{L}, \bsf)$ be an admissible $\Z$-representation of a quiver $Q$.  There exist $r, g, M \in \N$, rational functions $W_1, \dots, W_r \in \Q(X,Y_0, \boldsymbol{Y}_1, \boldsymbol{Y}_2)$, and a collection of $g$-ary formulae $\psi_1, \dots, \psi_r$ in the language of rings, such that the following holds for all compact discrete valuation rings $A$ of residue characteristic at least $M$ and for all $m \in \N$: for all $\boldsymbol{s} \in \CC^V$ such that $\mathrm{Re} \, s_v \gg_m 0$ for all $v \in V$ we have
$$ \zeta_{(\boldsymbol{L} \tensor_{\Z} A)^{\ast m}}(\boldsymbol{s}) = \left( \prod_{v \in V} \zeta^{\triangleleft A}_{A^{m n(v,1)}}(s_v) \right) \sum_{i = 1}^r m_i(k) W_i(q,q^m, (q^{-s_v}, q^{-ms_v})_{v \in V}),$$
where $k$ is the residue field of $A$, with $q = |k|$ and $m_i(k) = | \{ \xi \in k^g : k \models \psi_i(\xi) \} |$.  Moreover, there exists a universal bound $M = M_{\boldsymbol{n}}$ that works for all admissible $\Z$-representations $(\bsL, \bsf)$ of $Q$ with fixed rank vector $\boldsymbol{n}(\bsL,\bsf) = \boldsymbol{n}$.
\end{thm}

\subsection{Structure of the paper}
In Section~\ref{sec:preliminaries} below we obtain three auxiliary results on which the proofs of our main theorems rely.  Proposition~\ref{pro:rationality} states that certain $p$-adic integrals, over a fixed local field, are expressed by rational functions.
Proposition~\ref{pro:bdop} is a uniformity result for $p$-adic integrals of a particular type, whereas the crucial combinatorial Lemma~\ref{pro:sum.lattice} counts $A$-submodules $\Lambda$ of a direct product $L^m$ such that the sum of the projections of $\Lambda$ to the components of $L^m$ is a fixed submodule of $L$.  In Section~\ref{sec:proof} we carry out the program outlined in Section~\ref{sec:comparison} above by expressing the zeta functions of amalgamated powers of admissible quiver representations as $p$-adic integrals and applying the results of Section~\ref{sec:preliminaries} to prove Theorems~\ref{thm:quiver.local} and~\ref{thm:quiver.main}.

\section{Preliminaries} \label{sec:preliminaries}
\subsection{A rationality result}
Let $\mathrm{Loc}$ be the collection of pairs $(F,\pi_F)$, where $F$ is a non-Archimedean local field with normalized additive valuation $v_F$ (it suffices to consider one representative of each isomorphism class of such fields) and $\pi_F \in F$ is a uniformizer; we will generally omit $\pi_F$ in the notation.  Write $\mathcal{O}_F$ for the valuation ring of $F \in \mathrm{Loc}$, and let $k_F = \mathcal{O}_F / (\pi_F)$ be the residue field and $q_F = |k_F|$ its cardinality.  The multiplicative valuation on $F$ is given by $| x |_F = q_F^{-v_F(x)}$ for $x \in F$.  For every $n \in \N$, we denote by $\mu_F$ the Haar measure on $F^n$ normalized so that $\mu_F(\mathcal{O}_F^n) = 1$.  Given $M \in \N$, let $\mathrm{Loc}_{M} \subset \mathrm{Loc}$ be the collection of $F \in \mathrm{Loc}$ with $\mathrm{char} \, k_F \geq M$.  Denote by $\mathrm{Loc}^0 \subset \mathrm{Loc}$ the collection of local fields of characteristic zero.  

Let $F \in \mathrm{Loc}^0$.  A set $Y \subset F^\ell$ is \emph{semi-algebraic}~\cite[Definition~1.2]{Denef/86} if it may be constructed from sets of the form $\{ x \in F^\ell : \exists z \in F ,\  f(x) = z^n \}$, where $n \geq 2$ and $f \in F[x_1, \dots, x_\ell]$, by taking finitely many unions, intersections, and complements.  By Theorem~1.3 and Lemma~2.1 of~\cite{Denef/86}, a set $Y \subset F^\ell$ is semi-algebraic if and only if $Y$ is definable in Macintyre's language~\cite{Macintyre/76}.  A function is called semi-algebraic if its graph is a semi-algebraic set~\cite[Remark~1.5]{Denef/86}.

The following claim generalizes Denef's classical rationality theorem~\cite[Theorem~7.4]{Denef/84}. 
We say that $S \subset \R^n$ is a {\emph{generalized right half-space}} if $(s_1, \dots, s_n) \in S$ implies that $(s_1^\prime, \dots, s_n^\prime) \in S$ whenever $s_i^\prime \geq s_i$ for all $i \in [n]$.  If $\bss = (s_1, \dots, s_n) \in \CC^n$, we write $\mathrm{Re} \, \bss = (\mathrm{Re} \, s_1, \dots, \mathrm{Re} \, s_n) \in \R^n$.

\begin{pro} \label{pro:rationality}
Let $F \in \mathrm{Loc}^0$, let $Y \subset F^\ell$ be a semi-algebraic set, and consider semi-algebraic functions $f_0, f_1, \dots, f_n\colon Y \to F$. If there exists a generalized right half-space $S \subset \R^n$ such that the function
\[ I(s_1, \dots, s_n) = \int_Y |f_0(y) |_F |f_1 (y) |_F^{s_1} \cdots |f_n (y)|_F^{s_n} d \mu_F\]
is defined for all $\boldsymbol{s} = (s_1, \dots, s_n) \in \CC^n$ with $\mathrm{Re} \, \boldsymbol{s} \in S$, then $I(s_1, \dots, s_n)$ is a rational function in the variables $q_F^{- s_1}, \dots, q_F^{- s_n}$.  The denominator may be expressed as a product of factors of the form $(1 - q_F^{-a_0 - a_1 s_1 - \cdots - a_n s_n})$ with $(a_1, \dots, a_n) \in \N_0^{n}$ and $a_0 \in \Z$ such that $(a_0, \dots, a_n) \neq (0, \dots, 0)$.
\end{pro}
\begin{proof}
We imitate the proof of~\cite[Corollary~15]{CluckersLeenknegt/08}, which is our claim in the case $n = 1$.  For every $t \in \N$ set
\[ \mathcal{O}_F^{(t)} = \{ \pi_F^v (1 + \pi_F^t \alpha) : v \in \N_0, \alpha \in \mathcal{O}_F \} \subset \mathcal{O}_F.\]
By~\cite[Theorem~7]{CluckersLeenknegt/08} there is a finite partition of $Y$ into semi-algebraic parts such that each part $B$ either has measure zero (this corresponds to the map $g$ being a composition of bijections that include maps of the type $f_0$, in the notation of~\cite[Theorem~8]{CluckersLeenknegt/08}), or there exists a bijective semi-algebraic function $\varphi\colon (\mathcal{O}_F^{(t)})^\ell \to B$ such that, for every $0 \leq j \leq n$, we have
$v_F (f_j (\varphi (x))) = v_F \big( \beta_j \prod_{i = 1}^\ell x_i^{\mu_{ij}} \big)$ for some $\beta_j \in F$ and $(\mu_{1j}, \dots, \mu_{\ell j}) \in \Z^\ell$.  Moreover, the valuation of the Jacobian determinant of $\varphi$ has the same form.  Thus, possibly adjusting the integers $\mu_{i0}$ to absorb the Jacobian, we decompose $I(s_1, \dots, s_n)$ into a sum of finitely many integrals of the form
\[ 
\int_{(\mathcal{O}_F^{(t)} )^\ell} \left| \beta_0 \prod_{i = 1}^{\ell} x_i^{\mu_{i0}} \right|_{F}\ \prod_{j = 1}^n \left| \beta_j \prod_{i = 1}^\ell x_i^{\mu_{ij}} \right|_{F}^{s_j} d \mu_F.
\]
Since the integrand decomposes into a product of factors depending on a single variable $x_i$, we may assume without loss of generality that $\ell = 1$.  Then we obtain the integral
\begin{multline*}
\int_{\mathcal{O}_F^{(t)}} | \beta_0 x^{\mu_0} |_F \prod_{j = 1}^n |\beta_j x^{\mu_j} |_F^{s_j} d\mu_F (x) = \\
\sum_{v = 0}^\infty \mu_F(\pi_F^v(1 + \pi_F^t \mathcal{O}_F)) |\beta_0|_F \prod_{j = 1}^n | \beta_j|_F^{s_j} q_F^{-v(\mu_0 + \mu_1 s_1 + \cdots \mu_n s_n)} = \\
\frac{q_F^{-t} | \beta_0 |_F \prod_{j = 1}^n |\beta_j |^{s_j}_F}{1 - q_F^{-1 -\mu_0 - \sum_{j = 1}^n \mu_j s_j}} = \frac{q_F^{-t - v_F(\beta_0) - \sum_{j=1}^n v_F(\beta_j) s_j}}{1 - q_F^{-1 -\mu_0 - \sum_{j = 1}^n \mu_j s_j}}. 
\end{multline*}
Here the second equality follows by our assumptions regarding the convergence of $I(s_1, \dots, s_n)$.  Note that all terms of the series above are positive when $(s_1, \dots, s_n) \in \R^n$, so we must have $\mu_j \geq 0$ for all $j \in [n]$.  Moreover, if $\mu_j = 0$ for all $j \in [n]$, then $\mu_0 + 1 > 0$.  This completes the proof.
\end{proof}

\subsection{The Denef--Pas language and expansions} \label{sec:language.definition}
The Denef--Pas language $\mathfrak{L}_{\mathrm{DP}}$~\cite[Definition~2.3]{Pas/89} has three sorts: the valued field sort $\mathrm{VF}$ and the
residue field sort $\mathrm{RF}$ are endowed with the language of
rings, and the valued group sort $\mathrm{VG}$, which we will simply call
$\mathbb{Z}$, is endowed with the {Presburger language} of ordered abelian groups.
Moreover, $\mathfrak{L}_{\mathrm{DP}}$ has two function
symbols: $v\colon\mathrm{VF}\setminus\{0\}\rightarrow\mathrm{VG}$ and
$\mathrm{ac}\colon\mathrm{VF}\rightarrow\mathrm{RF}$, interpreted as a
valuation map and an angular component map, respectively.
Any formula $\phi$ in $\mathfrak{L}_{\mathrm{DP}}$ with $c_{1}$
free $\mathrm{VF}$-variables, $c_{2}$ free $\mathrm{RF}$-variables,
and $c_{3}$ free $\mathbb{Z}$-variables yields a subset $\phi(F)\subseteq F^{c_{1}}\times k_{F}^{c_{2}}\times\mathbb{Z}^{c_{3}}$
for any $F\in\mathrm{Loc}$.  A collection $X=(X_{F})_{F\in\mathrm{Loc}_{M}}$
with $X_{F}=\phi(F)$ and $M \in \N$ is called an \textsl{$\mathfrak{L}_{\mathrm{DP}}$-definable
set}.  An \textsl{$\mathfrak{L}_{\mathrm{DP}}$-}\textit{definable
function} is a collection of functions $f=(f_{F}\colon X_{F}\to Y_{F})_{F\in\mathrm{Loc}_{M}}$
such that the associated collection of graphs is an \textsl{$\mathfrak{L}_{\mathrm{DP}}$-}definable
set.  

In Sections~\ref{sec:restatement} and~\ref{sec:proof.quiver} below we will consider expansions of $\mathfrak{L}_{\mathrm{DP}}$ by a finite number of constant symbols.  Let $Q = (V,E)$ be a quiver and fix $\boldsymbol{n} \in (\N_0^2)^V$.
Let $\mathfrak{L}_{\boldsymbol{n}}$ be the expansion of the Denef--Pas language $\mathfrak{L}_{\mathrm{DP}}$ by $\sum_{e \in E} n(t(e),1) n(h(e),2)$ constant symbols $a_{ij}^e$, for $e \in E$ and $(i,j) \in [n(t(e),1)] \times [n(h(e),2)]$, of the valued field.  These will be interpreted as structure constants of a representation of the quiver $Q$.  Thus, a structure for the language $\mathfrak{L}_{\boldsymbol{n}}$ will be a triple $(F, \pi_F,\boldsymbol{a})$, where $(F, \pi_F) \in \mathrm{Loc}$, while the collection $\boldsymbol{a}$ of constants $a_{ij}^e$ gives an admissible $F$-representation $(\bsL, \bsf)$ with $F$-bases $b^v$ of $L_{v,1}$ and $\beta^v$ of $L_{v,2}$ for every $v \in V$ and arrows determined by $f_e(b_i^{t(e)}) = \sum_{j = 1}^{n(h(e),2)} a_{ij}^e \beta_j^{h(e)}$ for edges $e \in E$ and $i \in [n(t(e),1)]$.  Let $\mathrm{LocRep}$ be the set of such triples.  

\subsection{A uniformity result}
The rationality result of Proposition~\ref{pro:rationality} may be upgraded to a claim providing uniformity as $F$ varies over all local fields of sufficiently large residue characteristic.  

\begin{pro} \label{pro:bdop}
Let $Y = (Y_F)_F$ be an $\mathfrak{L}_{\mathrm{DP}}$-definable subset of $\mathrm{VF}^{c}$,
let $n\in\mathbb{N}$, and let $\Phi_0, \Phi_1, \dots, \Phi_{n}\colon Y \to\mathbb{Z}$
be $\mathfrak{L}_{\mathrm{DP}}$-definable functions.
Suppose that there exists a generalized right half-space $S \subset \R^n$ and $M \in \N$ such that for all $F \in \mathrm{Loc}_M$ the
integral $\int_{Y_{F}}q_{F}^{\Phi_0(y) + \Phi_{1}(y)s_1 + \cdots + \Phi_n(y) s_n }d\mu_F$
converges for all $\boldsymbol{s} = (s_1, \dots, s_n) \in\CC^n$ satisfying $\mathrm{Re} \, \boldsymbol{s} \in S$. 

Then there exist $\unifmax, g, M^\prime \in \N$, a collection of $g$-ary formulae $\psi_1, \dots, \psi_{\unifmax}$ in the language of rings, and rational functions
$W_{1},...,W_{\unifmax} \in\mathbb{Q}(X,Y_1, \dots, Y_n)$ such that the following holds for all $F\in\mathrm{Loc}_{M^\prime}$
and all $\boldsymbol{s} \in\CC^n$ satisfying $\mathrm{Re} \, \boldsymbol{s} \in S$:
\[
\int_{Y_{F}}q_{F}^{\Phi_0(y) + \Phi_{1}(y)s_1 + \cdots + \Phi_n(y) s_n}d\mu_F =\sum_{i=1}^{\unifmax}m_{i}(k_{F})\cdot W_{i}(q_{F},q_{F}^{-s_1}, \dots, q_F^{-s_n}),
\]
where $m_{i}(k_{F})$ is the cardinality of the set $\{\xi\in k_{F}^{g}:k_{F}\models \psi_{i}(\xi)\}$
for all $i\in[\unifmax]$. Moreover, for each $i \in [\unifmax]$ the denominator of $W_i$ may be expressed as a product of factors of the form $(1 - X^{a_0} Y_1^{a_1} \cdots Y_n^{a_n})$ for $(a_1, \dots, a_n) \in \N_0^n$ and $a_0 \in \Z$ such that $(a_0, \dots, a_n) \neq (0, \dots, 0)$.

\end{pro}
\begin{proof}
The case $n = 1$ is~\cite[Theorem~B]{BDOP/11}. 
The case $n = 2$ is~\cite[Theorem~4.1]{BGS/22}; its proof, like that of Proposition~\ref{pro:rationality}, uses rectilinearlization (\cite[Theorem~2]{Cluckers/03} and~\cite[Theorem~2.1.9]{CGH/14}).  Moreover, it generalizes in an obvious way to arbitrary $n \in \N$.
\end{proof}

\subsection{A combinatorial lemma}
The following observation is key to the method of this paper.  Let $A$ be a compact discrete valuation ring.
\begin{lem} \label{pro:sum.lattice}
Let $m,n \in \N$, and let $\Omega$ be a free $A$-module of rank $mn$.  Suppose that we are given a decomposition $\Omega = \Omega_1 \oplus \cdots \oplus \Omega_m$ and an $A$-module isomorphism $\varphi_i \colon \Omega_i \to A^n$ for every $i \in [m]$.  Let $\pi_i \colon \Omega \to \Omega_i$ be the projection onto the $i$-th component.  For every finite-index $A$-submodule $H \leq A^n$, let $S_H$ be the set of finite-index $A$-submodules $\Lambda \leq \Omega$ such that $\sum_{i = 1}^m \varphi_i (\pi_i(\Lambda)) = H$.  Then
\[ \sum_{\Lambda \in S_H} [\Omega : \Lambda]^{-s} = \frac{\zeta^{\triangleleft A}_{A^{mn}}(s)}{\zeta^{\triangleleft A}_{A^n}(ms)} [A^n : H]^{-ms}.\]
\end{lem}
\begin{proof}
Fix a uniformizer $\pi \in A$.
There exists an $A$-basis $(\beta_1,  \dots, \beta_n)$ of $A^n$ and integers $\nu_1,  \dots, \nu_n \in \N_0$ such that $H$ is the $A$-linear span of $\{ \pi^{\nu_1} \beta_1, \dots, \pi^{\nu_n} \beta_n \}$.  Set $\beta_j^{(i)} = \varphi_i^{-1}(\beta_j)$ for every $i \in [m]$ and $j \in [n]$.  Then $\{ \beta^{(i)}_j : (i,j) \in [m] \times [n] \}$ is an $A$-basis of $\Omega$.  Consider the $A$-module endomorphism $\psi\colon \Omega \to \Omega$ determined by $\psi(\beta^{(i)}_j) = \pi^{\nu_j} \beta^{(i)}_j$.  Let $\widetilde{\psi}$ be the endomorphism of $A^n$ given by $\widetilde{\psi}(\beta_j) = \pi^{\nu_j} \beta_j$.  For any finite-index $A$-submodule $\Lambda \leq \Omega$, it is clear that
\[ \sum_{i = 1}^m \varphi_i(\pi_i(\psi(\Lambda))) = \widetilde{\psi} \left( \sum_{i = 1}^m \varphi_i (\pi_i (\Lambda)) \right).\]
In particular, since $\widetilde{\psi}(A^n) = H$, the map $\psi$ induces a map $S_{A^n} \to S_H$.  This map is injective since $\psi$ is.  Moreover, it is easy to see that if $\Lambda \leq H$, then there exists an $A$-submodule $\Lambda^\prime \leq \Omega$ satisfying $\Lambda = \psi(\Lambda^\prime)$.  Thus $\psi$ induces a bijection between $S_{A^n}$ and $S_H$.
Moreover,
\[ [\Omega: \psi(\Lambda)] = q^{\sum_{i = 1}^m \sum_{j = 1}^n \nu_j} [\Omega: \Lambda] = q^{m \sum_{j = 1}^n \nu_j} [\Omega: \Lambda] = [A^n : H]^m [\Omega: \Lambda]\]
for any $A$-submodule $\Lambda \leq \Omega$.  Hence
\begin{equation} \label{equ:an.h}
 \sum_{\Lambda \in S_H} [\Omega: \Lambda]^{-s} = [A^n : H]^{-ms} \sum_{\Lambda \in S_{A^n}} [\Omega: \Lambda]^{-s}.
\end{equation}
We conclude that
\begin{multline} \label{equ:total}
\zeta^{\triangleleft A}_{A^{mn}}(s) = \sum_{\Lambda \leq \Omega} [\Omega: \Lambda]^{-s} = \sum_{H \leq A^n} \sum_{\Lambda \in S_H} [\Omega: \Lambda]^{-s} = \\
\left( \sum_{H \leq A^n} [A^n : H]^{-ms} \right) \left( \sum_{\Lambda \in S_{A^n}} [\Omega: \Lambda]^{-s} \right) = \zeta^{\triangleleft A}_{A^n}(ms) \sum_{\Lambda \in S_{A^n}} [\Omega: \Lambda]^{-s}.
\end{multline}
where the sums run over finite-index submodules. The claim is immediate from~\eqref{equ:an.h} and~\eqref{equ:total}.
\end{proof}

\begin{rem}
In the case $m = 1$, Lemma~\ref{pro:sum.lattice} is a special case of~\cite[Proposition~4.2]{CSV/19}, which is proved by a similar argument.  A mutual generalization of Lemma~\ref{pro:sum.lattice} and~\cite[Proposition~4.2]{CSV/19} is readily obtained, but we do not state it as we shall not need it in this paper.
\end{rem}

\section{Quiver representation zeta functions and \texorpdfstring{$p$}{p}-adic integration} \label{sec:proof}
Fix a quiver $Q = (V,E)$. 
Throughout this section, $A$ is a compact discrete valuation ring with residue field of cardinality $q$.

\subsection{Rewriting the zeta function}
Let $(\boldsymbol{L}, \boldsymbol{f})$ be an admissible $A$-representation of the quiver $Q$.  For $i \in [2]$, define $\bsL_i$ to be the family $(L_{v,i})_{v \in V}$.  Given two families $\bsH = (H_v)$ and $\bsH^\prime = (H^\prime_v)$, we write $\bsH \leq \bsH^\prime$ if $H_v \leq H^\prime_v$ is a finite-index $A$-submodule for all $v$.  For $\bsH \leq \bsL_1$ and $v \in V$, set 
$$\bsf(\bsH)_v = \sum_{\substack{e \in E \\ h(e) = v}} f_e (H_{t(e)}) \leq L_{v,2}.$$
This defines a family $\bsf(\bsH) \leq \bsL_2$.
Recall the rank vector $\boldsymbol{n}(\bsL, \bsf)$ from Definition~\ref{def:admissible.rep}.
The following claim, for which admissibility is essential, is the analogue for quiver representations of~\cite[Lemma~6.1]{GSS/88}.

\begin{lem} \label{lem:lemma6.1}
Let $(\bsL,\bsf)$ be an admissible $A$-representation of $Q$ with rank vector $\boldsymbol{n}$.  Then
$$ \zeta_{\bsL}(\bss) = \sum_{\bsH \leq \bsL_1} \sum_{\substack{\bsM \leq \bsL_2 \\ \bsf(\bsH) \leq \bsM}} \left( \prod_{v \in V} [L_{v,1} : H_v]^{-s_v} [L_{v,2} : M_v]^{n(v,1) - s_v} \right).$$
\end{lem}
\begin{proof}
A square matrix with entries in $A$ is said to be in Hermite normal form if it is upper triangular, the diagonal elements are powers of $\pi$, and the entries in a column above a diagonal element $\pi^j$ lie in a fixed set of coset representatives of $A / \pi^j A$.
For every $v \in V$, let $b_{v,1}$ and $b_{v,2}$ be $A$-bases of $L_{v,1}$ and $L_{v,2}$, respectively.  Concatenate them to obtain an $A$-basis of $L_v$.  
For every finite-index $A$-submodule $\Lambda_v \leq L_v$,  there exists a unique matrix $\mathrm{Mat}(\Lambda_v)$ in Hermite normal form such that $\Lambda_v$ is spanned by the rows of $\mathrm{Mat}(\Lambda_v)$, viewed as elements of $L_v$ with respect to the chosen basis; see, for instance,~\cite[Theorem~22.1]{MacDuffee/33}.  Consider $\mathrm{Mat}(\Lambda_v)$ as a block matrix
$$ \mathrm{Mat}(\Lambda_v) = \left( \begin{array}{cc} \mathrm{Mat}(H_v) & C_v \\ 0 & \mathrm{Mat}(M_v) \end{array} \right),$$
where $C_v$ is an $n(v,1) \times n(v,2)$ matrix with entries in $A$, while $H_v \leq L_{v,1}$ and $M_v \leq L_{v,2}$ are finite-index $A$-submodules.  It is clear that $\boldsymbol{\Lambda} = (\Lambda_v)_{v \in V}$ is an $A$-subrepresentation of $\bsL$ if and only if $\bsf(\bsH) \leq \bsM$.  Clearly $[L_v : \Lambda_v] = [L_{v,1} : H_v] [L_{v,2} : M_v]$, and, given $H_v$ and $M_v$, the Hermite normal form admits $[L_{v,2} : M_v]$ possible choices for each of the $n(v,1)$ rows of $C_v$.  The claim follows.
\end{proof}

The next proposition allows us to rewrite the zeta function $\zeta_{\bsL^{\ast m}}(\bss)$ as a product of an explicit factor and an infinite sum in which only the summands depend on $m$, but the set parametrizing them is independent of $m$.

\begin{pro} \label{pro:rewrite}
Let $(\bsL,\bsf)$ be an admissible $A$-representation of $Q$ with rank vector $\boldsymbol{n}$, and let $m \in \N$.  Then
\[ \mathclap{\zeta_{\bsL^{\ast m}}(\bss) = \left( \prod_{v \in V} \frac{\zeta^{\triangleleft A}_{A^{mn(v,1)}}(s_v)}{\zeta^{\triangleleft A}_{A^{n(v,1)}}(ms_v)} \right)\! \sum_{\bsH \leq \bsL_1} \sum_{\substack{\bsM \leq \bsL_2 \\ \bsf(\bsH) \leq \bsM}} \!\left( \prod_{v \in V} [L_{v,1} : H_v]^{-ms_v} [L_{v,2} : M_v]^{mn(v,1) - s_v} \right) .}\]
\end{pro}
\begin{proof}
For $v \in V$ and $i \in [m]$, let $\pi_{v,i} \colon L_{v,1}^m \to L_{v,1}$ be the projection onto the $i$-th component and let $\varphi_{v,i}\colon \pi_{v,i}(L_{v,1}^m) \stackrel{\sim}{\to} L_{v,1}$ be the natural identification, so that for each $v$ we are in the setup of Lemma~\ref{pro:sum.lattice} with $\Omega= (L^{\ast m})_{v,1} = L_{v,1}^m$.
By Lemma~\ref{lem:lemma6.1} we have
\begin{equation} \label{equ:gss.lemma}
    \zeta_{\bsL^{\ast m}}(\bss) = \sum_{\boldsymbol{\Lambda} \leq \bsL_1^m} \sum_{\substack{\bsM \leq \bsL_2 \\ \bsf^{\ast m}(\boldsymbol{\Lambda}) \leq \bsM}} \left( \prod_{v \in V} [L_{v,1}^m : \Lambda_v]^{-s_v} [L_{v,2} : M_v]^{mn(v,1) - s_v} \right) .
\end{equation}
Given $\boldsymbol{\Lambda} \leq \bsL_1^m$, set $\boldsymbol{S}(\boldsymbol{\Lambda})_v = \sum_{i = 1}^m \varphi_{v,i}(\pi_{v,i}(\Lambda_v))$ for every $v \in V$.  Observe that
$$ \bsf^{\ast m}(\boldsymbol{\Lambda})_v = \sum_{\substack{e \in E \\ h(e) = v}} \sum_{i = 1}^m f_{e,i}(\Lambda_{t(e)}) = \sum_{\substack{e \in E \\ h(e) = v}} f_e \left( \sum_{i = 1}^m \varphi_{t(e),i}(\pi_{t(e),i}(\Lambda_{t(e)})) \right) = \bsf(\boldsymbol{S}(\boldsymbol{\Lambda}))_v.$$

Hence the right-hand side of~\eqref{equ:gss.lemma} may be expressed as 
\[ 
    \sum_{\bsH \leq \bsL_1} \sum_{\substack{\bsM \leq \bsL_2 \\ \bsf(\bsH) \leq \bsM}} \sum_{\substack{\boldsymbol{\Lambda} \leq \bsL_1^m \\ \boldsymbol{S}(\boldsymbol{\Lambda}) = \bsH}} \left( \prod_{v \in V} [L_{v,1}^m : \Lambda_v]^{-s_v} [L_{v,2} : M_v]^{mn(v,1) - s_v} \right).
\]
Our claim follows from Lemma~\ref{pro:sum.lattice}, since $\boldsymbol{S}(\boldsymbol{\Lambda}) = \bsH$ if and only if $\Lambda_v \in S_{H_v}$ for all $v \in V$ in the notation of that lemma.
\end{proof}

\subsection{Restatement in terms of \texorpdfstring{$p$}{p}-adic integrals} \label{sec:restatement}
Let $F$ be the fraction field of $A$, and fix a uniformizer $\pi \in A$.  Then $(F,\pi) \in \mathrm{Loc}$ and $A = \mathcal{O}_F$.  
Let $(\bsL, \bsf)$ be an admissible $A$-representation of $Q$ with rank vector $\boldsymbol{n}$, and consider the infinite sum 
\begin{equation} \label{equ:quantity}
    \sum_{\bsH \leq \bsL_1} \sum_{\substack{\bsM \leq \bsL_2 \\ \bsf(\bsH) \leq \bsM}} \left( \prod_{v \in V} [L_{v,1} : H_v]^{-ms_v} [L_{v,2} : M_v]^{mn(v,1) - s_v} \right),
\end{equation}
which by Proposition~\ref{pro:rewrite} is equal to $\zeta_{\bsL^{\ast m}} (\bss)$ up to an explicit factor. Our aim is to express it as a $p$-adic integral to which the results of Section~\ref{sec:preliminaries} are applicable.  This type of argument, for sums running over subgroups, was introduced in~\cite[\S2]{GSS/88}.  

\begin{dfn}
Let $U$ be a free $A$-module of rank $r$ with a fixed basis $(b_1, \dots, b_r)$.  Set $U_i = \langle b_i, \dots, b_r \rangle_{A}$ for every $i \in [r+1]$.  In particular, $U_{r + 1} = 0$.  
If $\mathcal{M} \leq U$ is a finite-index $A$-submodule, we say that an $A$-basis $(c_1, \dots, c_r)$ of $\mathcal{M}$ is a {\emph{good basis}} if $\langle c_i, \dots, c_r \rangle_{A} = \mathcal{M} \cap U_i$ for every $i \in [r]$.
\end{dfn}
Consider the additive algebraic group $T_r$ of upper triangular $r \times r$ matrices, defined over $\Z$.  For any $F \in \mathrm{Loc}$ we identify $T_r (F)$ with $F^{\binom{r+1}{2}}$ in the natural way and let $\mu_F$ be the additively invariant measure on $T_r (F)$ normalized so that $\mu_F(T_r (\mathcal{O}_F)) = 1$.  For $\mathcal{M} \leq U$ as above, let $T(\mathcal{M}) \subset T_r (A)$ be the set of all matrices $C \in T_r (A)$ such that the rows of $C$, interpreted as elements of $U$ with respect to the $A$-basis $(b_1, \dots, b_r)$, constitute a good basis of $\mathcal{M}$.  The following claim is the analogue in our setting of~\cite[Lemma~15.1.1]{LubotzkySegal/03}; see also~\cite[Lemma~2.5]{GSS/88}.

\begin{lem} \label{lem:measure}
Let $U$ be a free $A$-module of rank $r$, and let $\mathcal{M} \leq U$ be an $A$-submodule of finite index.  Then $T(\mathcal{M})$ is an open subset of $T_r (A)$ and 
\begin{equation*} 
 \mu_F (T(\mathcal{M})) = (1 - q^{-1})^r q^{- \sum_{i = 1}^r i \lambda_i},
\end{equation*}
where $\lambda_i$, for all $i \in [ r ]$, is determined by
$ q^{\lambda_i} = [ U_i  : (\mathcal{M} \cap U_i ) + U_{i + 1}]$.  Moreover, for any $C = (c_{ij} ) \in T(\mathcal{M})$ the following holds:
\begin{equation*}
[U  : \mathcal{M}] = q^{\lambda_1 + \dots + \lambda_r} = \prod_{i = 1}^r |c_{ii} |^{-1}_F .
\end{equation*}
\end{lem}
\begin{proof}
We construct $T(\mathcal{M})$ inductively from the bottom row up.  For any $i \in [ r ]$ the $i$-th row of a matrix $C \in T(\mathcal{M})$ necessarily has a leading term of additive valuation $v_F(c_{ii}) = \lambda_i$.  This implies the second part of our claim, since $[U  : \mathcal{M}] = | \det C |^{-1}_F$ for any $C \in T(\mathcal{M})$.  If the lower rows of $C$ have been constructed already, then the $i$-th row is well-defined modulo addition of $A$-multiples of the lower rows.  Thus the set of possible matrix entries $c_{ii}$ has measure $(1 - q^{-1})q^{- \lambda_i}$ in $F$, whereas the set of possible $(r-i)$-tuples $(c_{i,i+1}, \dots, c_{ir})$ has measure $q^{-(\lambda_{i+1} + \cdots + \lambda_r)}$ in $F^{r-i}$.
\end{proof}

For every $v \in V$, fix $A$-bases $b^v = (b^{v}_1, \dots, b^v_{n(v,1)})$ of ${L}_{v,1}$ and $\beta^v = (\beta^v_{1}, \dots, \beta^v_{n(v,2)})$ of $L_{v,2}$.  Having chosen these bases, given families $\bsH \leq \bsL_1$ and $\bsM \leq \bsL_2$ 
we can define good bases of $H_v$ and $M_v$ for every $v \in V$.  For every edge $e \in E$ and for every $(i,j) \in [n(t(e),1)] \times [n(h(e),2)]$ let $a^e_{ij} \in A$ be the structure constants satisfying $f_e(b^{t(e)}_i) = \sum_{j = 1}^{n(h(e),2)} a^e_{ij} \beta^{h(e)}_j$.

\begin{dfn} \label{def:good.basis}
Let
$Y_{\bsL}$ be the set of tuples
\[(C^{v,1}, C^{v,2})_v \in \prod_{v \in V} (T_{n(v,1)}(F) \times T_{n(v,2)}(F)) = \prod_{v \in V} \left( F^{\binom{n(v,1)+1}{2} + \binom{n(v,2)+1}{2}} \right)
\]
such that the $A$-linear spans of the rows of $C^{v,1}=(c_{ij}^{v,1})$ and $C^{v,2}=(c_{ij}^{v,2})$, with respect to the bases $b^v$ and $\beta^v$, are finite-index $A$-submodules $H_v \leq L_{v,1}$ and $M_v \leq L_{v,2}$, respectively, satisfying $\bsf(\bsH) \leq \bsM$.  Equivalently, $(C^{v,1}, C^{v,2})_v  \in Y_{\bsL}$ if and only if the following conditions are satisfied:
\begin{itemize}
\item $v_F(c^{v,\ell}_{ij}) \geq 0$ for all $v \in V$ and $\ell \in \{1,2 \}$ and $1 \leq i \leq j \leq n(v,\ell)$.
\item $\det C^{v,\ell} \neq 0$ for all $v \in V$ and $\ell \in \{ 1, 2 \}$.
\item For all $e \in E$ and $(i, j) \in [n(t(e),1)] \times [n(h(e),2)]$ there exists $d_{ij}^e \in F$ such that $v_F (d_{ij}^e) \geq 0$ and the following holds for all $(i, k) \in [n(t(e),1)] \times [n(h(e),2)]$:
\begin{equation} \label{equ:expanded.conditions}
\sum_{\ell = i}^{n(t(e),1)} c_{i \ell}^{t(e),1} a_{\ell k}^e = \sum_{j = 1}^k d_{ij}^e c_{jk}^{h(e),2}.
\end{equation}
\end{itemize}
\end{dfn}

It is clear that the conditions of Definition~\ref{def:good.basis} depend only on the quiver $Q$, the rank vector $\boldsymbol{n}(\bsL)$, and the structure constants $a_{ij}^e \in A$.  Hence they can be expressed by an $\mathfrak{L}_{\boldsymbol{n}}$-formula $\psi((C^{v,1}, C^{v,2})_v)$, where $\boldsymbol{n} = \boldsymbol{n}(\bsL)$ and $\mathfrak{L}_{\boldsymbol{n}}$ is the expanded Denef--Pas language of Section~\ref{sec:language.definition}.
Observe that
\begin{equation} \label{equ:yl}
Y_{\bsL} = \bigcup_{\bsH \leq \bsL_1} \bigcup_{\substack{\bsM \leq \bsL_2 \\ \bsf(\bsH) \leq \bsM}} \prod_{v \in V} T(H_v) \times T({M_v}),
\end{equation}
where the union is disjoint.  Thus the expression~\eqref{equ:quantity} may be rewritten as
\begin{equation} \label{equ:unif.integral}
   \sum_{\bsH \leq \bsL_1} \sum_{\substack{\bsM \leq \bsL_2 \\ \bsf(\bsH) \leq \bsM}} \prod_{v \in V} \int_{T(H_v) \times T({M_v})} \frac{[L_{v,1} : H_v]^{-ms_v} [L_{v,2} : M_v]^{mn(v,1) - s_v}}{\mu_F(T(H_v)) \cdot \mu_F(T({M_v}))} d \mu_F.
\end{equation}

\begin{lem} \label{lem:final.integral}
Let $(\bsL,\bsf)$ be an admissible $A$-representation of $Q$ with rank vector $\boldsymbol{n}$.  Then
\begin{multline*}
\zeta_{\bsL^{\ast m}} (\bss) =  \prod_{v \in V} \left( \frac{\zeta^{\triangleleft A}_{A^{mn(v,1)}}(s_v)}{\zeta^{\triangleleft A}_{A^{n(v,1)}}(ms_v)}
(1 - q^{-1})^{-(n(v,1) + n(v,2))} \right) \times \\
\int_{Y_{\bsL}} \prod_{v \in V} \left( \prod_{i = 1}^{n(v,1)} |c_{ii}^{v,1} |_{F}^{ms_v - i} \right) \left( \prod_{j = 1}^{n(v,2)} |c^{v,2}_{jj} |_{F}^{s_v - mn(v,1) - j} \right) d \mu_F.
\end{multline*}
\end{lem}
\begin{proof}
The equality follows from Proposition~\ref{pro:rewrite} and the observation~\eqref{equ:yl}, after substituting the claims of Lemma~\ref{lem:measure} into~\eqref{equ:unif.integral}.
\end{proof}

\subsection{Proof of Theorem~\ref{thm:quiver.main}} \label{sec:quiver.main.proof}
Let $(\bsL,\bsf)$ be an admissible $\Z$-representation of $Q$.  For every $v \in V$, fix $\Z$-bases $b^v$ of $L_{v,1}$ and $\beta^v$ of $L_{v,2}$.  Recall the structure constants $a_{ij}^e \in \Z$ for $e \in E$ and $(i,j) \in [n(t(e),1)] \times [n(h(e),2)]$.  For every $F \in \mathrm{Loc}$, the $\mathcal{O}_F$-representation $\bsL \tensor_{\Z} \mathcal{O}_F$ is determined by the constants $a_{ij}^e$, interpreted as elements of $\mathcal{O}_F$.  Set 
$$Y_F = Y_{\bsL \tensor_{\Z} \mathcal{O}_F} = \left\{ (C^{v,1}, C^{v,2})_v \in \prod_{v \in V} \left( F^{\binom{n(v,1)+1}{2} + \binom{n(v,2)+1}{2}} \right) : F \models \psi((C^{v,1}, C^{v,2})_v) \right\}$$
as in Definition~\ref{def:good.basis}.  The collection $Y = (Y_F)_{F \in \mathrm{Loc}}$ is then an $\mathfrak{L}_{\mathrm{DP}}$-definable set; note that the structure constants $a_{ij}^e$ appearing in~\eqref{equ:expanded.conditions} are fixed integers and may be expressed by repeated addition, so that the expanded language $\mathfrak{L}_{\boldsymbol{n}}$ is not needed in this case.

We treat $m$ as a formal complex variable and apply Proposition~\ref{pro:bdop} to the statement of Lemma~\ref{lem:final.integral}, where we take $n = 2 |V| + 1$ and the variables $s_1, \dots, s_n$ to be $-m$ in addition to $s_v$ and $ms_v$ for all $v \in V$.  To see that the hypotheses of Proposition~\ref{pro:bdop} are satisfied, set $O = \mathcal{O}_F^{\binom{n(v,1) + 1}{2} + \binom{n(v,2) + 1}{2}}$ and note that the integrand of the integral in Lemma~\ref{lem:final.integral} is positive if the $s_v$ are real.  In this case the integral is bounded by
\[
\int_{O} \prod_{v \in V} \left( \prod_{i = 1}^{n(v,1)} |c_{ii}^{v,1} |_{F}^{ms_v - i} \right) \left( \prod_{j = 1}^{n(v,2)} |c^{v,2}_{jj} |_{F}^{s_v - mn(v,1) - j} \right) d \mu_F,
\]
which converges when $ms_v - n(v,1) > 0$ and $s_v - mn(v,1) - n(v,2) > 0$ for all $v \in V$.  Thus the integral converges for all $\bss \in \CC^{2 |V| + 1}$ such that $\mathrm{Re} \, \bss$ lies in the generalized half-space cut out by the conditions $\mathrm{Re} \, ms_v > n(v,1)$ and $\mathrm{Re} \, s_v + n(v,1) \mathrm{Re} \, (-m) > n(v,2)$.  These are satisfied, for instance, when $\mathrm{Re} \, s_v > m n(v,1) + n(v,2)$ for all $v \in V$.

By Proposition~\ref{pro:bdop} we conclude that there exists $M \in \N$, depending on $(\bsL,\bsf)$, such that for all $F \in \mathrm{Loc}_M$ we have 
\begin{multline} \label{equ:intermediate}
\zeta_{(\bsL \tensor_{\Z} \mathcal{O}_F)^{\ast m}} (\bss) = \\
\prod_{v \in V} \left( \frac{\zeta^{\triangleleft \mathcal{O}_F}_{\mathcal{O}_F^{mn(v,1)}}(s_v)}{\zeta^{\triangleleft \mathcal{O}_F}_{\mathcal{O}_F^{n(v,1)}}(ms_v)} \cdot
    \frac{q_F^{n(v,1) + n(v,2)}}{(q_F - 1)^{n(v,1) + n(v,2)}} \right) \cdot \sum_{i=1}^{\unifmax}m_{i}(k_{F})\cdot \widetilde{W}_{i}(q_{F},q_F^m,(q_{F}^{-s_v},q_F^{-ms_v})_{v \in V})
 \end{multline}
for rational functions $\widetilde{W}_1, \dots, \widetilde{W}_{\unifmax} \in \Q(X,Y_0,(Y_{1,v},Y_{2,v})_{v \in V})$ and for $g$-ary formulae $\psi_1, \dots, \psi_{\unifmax}$, where $m_i(k_F) = | \{ \xi \in k_F^g : k_F \models \psi_i(\xi) \} |$.  
Here $q_F$ is the cardinality of the residue field $k_F$ of $F$.
Set
\[W_i(X,Y_0, \boldsymbol{Y}_1, \boldsymbol{Y}_2) = \widetilde{W}_i(X,Y_0, \boldsymbol{Y}_1, \boldsymbol{Y}_2) \prod_{v \in V} \left( \!\left( \frac{-X}{1-X} \right)^{n(v,1) + n(v,2)} \cdot \!\prod_{j = 0}^{n(v,1)-1} \!(1 - X^j Y_{2,v}) \right)\!.\] 
Recall that any complete discrete valuation ring with finite residue field is the valuation ring of a local field.  We deduce from~\eqref{equ:abelian} and~\eqref{equ:intermediate} that these rational functions, and the formulae $\psi_1, \dots, \psi_r$, satisfy the statement of Theorem~\ref{thm:quiver.main}.  

It remains to show that the bound $M$ on the residue characteristic may be chosen in a way that depends only on the rank vector $\boldsymbol{n}$; this will be done in Remark~\ref{rmk:nguyen.to.main} below.

\begin{rem}
The argument above relies on Proposition~\ref{pro:bdop}, whose proof in turn uses relatively recent uniformity results as in~\cite{CGH/14} to obtain a statement holding for all compact discrete valuation rings $A$ of sufficiently large residue characteristic.  Weaker versions of the theorem can be obtained by the same argument, using older uniformity results; for instance, the original work of Denef~\cite{Denef/84}, as well as~\cite{Macintyre/90, Pas/89}, considered $A = \Z_p$ for sufficiently large $p$, while~\cite{CluckersLoeser/15} considered all $A$ of characteristic zero with bounded ramification.
\end{rem}

\subsection{Proof of Theorem~\ref{thm:quiver.local}} \label{sec:proof.quiver}
First let $A$ be a compact discrete valuation ring of characteristic zero, and let
$F \in \mathrm{Loc}^0$ be its fraction field.  Let $(\bsL,\bsf)$ be an admissible $A$-representation of $Q$.  The domain of integration $Y_{\bsL}$ of Definition~\ref{def:good.basis} is semi-algebraic by~\cite[Lemma~2.1]{Denef/86}.  Hence Theorem~\ref{thm:quiver.local} follows in this case by applying Proposition~\ref{pro:rationality} to the statement of Lemma~\ref{lem:final.integral} and proceeding analogously to the end of the proof of Theorem~\ref{thm:quiver.main}.

In order to treat equal-characteristic local fields of sufficiently large characteristic, we need a bit more technology.  

Recall the expanded Denef--Pas language $\mathfrak{L}_{\boldsymbol{n}}$ and the collection $\mathrm{LocRep}$ from Section~\ref{sec:language.definition}.  
Definition~\ref{def:good.basis}
makes sense for any collection $\boldsymbol{a}$ of constants $a_{ij}^e \in F$ and thus gives an $\mathfrak{L}_{\boldsymbol{n}}$-definable set $\{ Y_{(F, \pi_F, \boldsymbol{a})} \}_{(F, \pi_F, \boldsymbol{a}) \in \mathrm{LocRep}}$.  The sets $Y_{\bsL}$ of Definition~\ref{def:good.basis} are the members of this collection parametrized by triples $(F, \pi_F, \boldsymbol{a})$ such that all elements of $\boldsymbol{a}$ lie in the valuation ring $A$ of $F$.

Given $(F,\pi_F,\boldsymbol{a}) \in \mathrm{LocRep}$, let $(\widetilde{\bsL}, \widetilde{\bsf})$ be the associated $F$-representation of $Q$.  For every $v \in V$, let $L_{v,1}$ be the $A$-submodule of $\widetilde{L}_{v,1}$ generated by the basis $b^v$, and define $L_{v,2}$ analogously.
Consider the $\mathfrak{L}_{\boldsymbol{n}}$-formula $\varphi \left( (C^{v,1}, C^{v,2}, D^{v,1}, D^{v,2}, t_{v,1},t_{v,2})_{v \in V}, u \right)$, where $u \in \Z$ and $t_{v,i} \in \Z$, while $C^{v,i},D^{v,i} \in F^{\binom{n(v,i)+1}{2}}$ for $v \in V$ and $i \in [2]$, expressing the following:
 there exist $\boldsymbol{H} \leq \bsL_1$ and $\boldsymbol{M} \leq \bsL_2$ such that:
\begin{itemize}
\item
$[L_{v,1} : H_v] = q_F^{t_{v,1}}$ for all $v \in V$;
\item
$[L_{v,2} : M_v] = q_F^{t_{v,2}}$ for all $v \in V$;
\item
$\prod_{v \in V} [L_{v,2}: M_v]^{n(v,1)} = q_F^u$;
\item
$(C^{v,1}, C^{v,2})_v, (D^{v,1}, D^{v,2})_v \in \prod_{v \in V} (T(H_v) \times T(M_v)) \subset Y_{(F,\pi_F,\boldsymbol{a})}$.
\end{itemize}
More precisely, $\varphi$ may be taken to be the conjunction of the conditions:
\begin{itemize}
\item
$(C^{v,1}, C^{v,2})_v, (D^{v,1}, D^{v,2})_v \in Y_{(F,\pi_F,\boldsymbol{a})}$;
\item
$\sum_{i = 1}^{n(v,\ell)} v_F(c_{ii}^{v,\ell}) = \sum_{i = 1}^{n(v,\ell)} v_F(d_{ii}^{v,\ell}) = t_{v,\ell}$ for all $v \in V$ and $\ell \in [2]$;
\item
$u = \sum_{v \in V} n(v,1) t_{v,2}$;
\item
For every $v \in V$ and $\ell \in [2]$ there exists $B^{v,\ell} \in T_{n(v,\ell)}(\mathcal{O}_F)$ with unit diagonal entries such that $C^{v,\ell} = B^{v,\ell} D^{v,\ell}$; this condition is obviously $\mathfrak{L}_{\mathrm{DP}}$-definable.
\end{itemize}

The formula $\varphi$ affords an $\mathfrak{L}_{\boldsymbol{n}}$-definable equivalence relation.  Set $\boldsymbol{t}_i = (t_{v,i})_{v \in V}$ for $i \in [2]$, and let $\delta_{\varphi, \boldsymbol{a}, \boldsymbol{t}_1, \boldsymbol{t}_2, u}$ be the number of equivalence classes.  This is the number of pairs $(\bsH,\bsM)$, where $\bsH \leq \bsL_1$ and $\bsM \leq \bsL_2$ and $\widetilde{\bsf}(\bsH) \leq \bsM$, such that $[L_{v,1} : H_v] = q_F^{t_{v,1}}$ and $[L_{v,2} : M_v] = q_F^{t_{v,2}}$ for all $v \in V$.  
If all the structure constants $a_{ij}^e$ lie in $A$, then $(\bsL, \widetilde{\bsf}_{| \bsL})$ is an admissible $A$-representation of $Q$, and all admissible $A$-representations arise in this way.

By~\cite[Theorem~4.1.1]{Nguyen/19}, whose hypotheses hold by~\cite[Example~2.3.7]{Nguyen/19}, there exists $M_{\boldsymbol{n}} \in \N$ such that the Poincar\'{e} series 
$$\sum_{(\boldsymbol{t}_1, \boldsymbol{t}_2, u) \in \N_0^{2|V| + 1}} \delta_{\varphi, \boldsymbol{a}, \boldsymbol{t}_1, \boldsymbol{t}_2, u} T_0^{u} T_{1,v}^{t_{v,1}} T_{2,v}^{t_{v,2}}$$
may be expressed as a rational function $\widetilde{W}_{\boldsymbol{a}} \in \Q(T_0, (T_{1,v}, T_{2,v})_{v \in V})$ with denominator of the requisite form provided that the residue characteristic of $F$ is at least $M_{\boldsymbol{n}}$.  Note that we absorb the cardinalities denoted by $\# Y(k_K)$ in the statement of~\cite[Theorem~4.1.1]{Nguyen/19} into the rational coefficients of the numerator.  Alternatively, Theorem~6.1 and Corollary~6.8 of~\cite{HMR/18} may be used in place of~\cite[Theorem~4.1.1]{Nguyen/19}, with the observation that, by the Ax--Kochen--Ershov principle, the proof of~\cite[Corollary~6.8]{HMR/18} applies to all local fields of sufficiently large positive characteristic. 
If $(\bsL, \bsf)$ is an admissible $A$-representation of $Q$ with structure constants $\boldsymbol{a}$, then we write $\widetilde{W}_{\bsL}$ for $\widetilde{W}_{\boldsymbol{a}}$.
Then 
$$\zeta_{\bsL^{\ast m}} (\bss) = \prod_{v \in V} \left( \frac{\zeta^{\triangleleft A}_{A^{mn(v,1)}}(s_v)}{\zeta^{\triangleleft A}_{A^{n(v,1)}} (ms_v)} \right) \widetilde{W}_{\bsL}(q^m, (q^{-s_v}, q^{-ms_v})_{v \in V})$$
by Proposition~\ref{pro:rewrite}, from which our claim follows.

\begin{rem} \label{rmk:nguyen.to.main}
Theorem~\ref{thm:quiver.main} may also be deduced via an application of~\cite[Theorem~4.1.1]{Nguyen/19} as in the proof of Theorem~\ref{thm:quiver.local}.  In fact, such an argument provides the additional information that the bound $M$ in the statement of Theorem~\ref{thm:quiver.main} depends only on the quiver $Q$ and the rank vector $\boldsymbol{n}$ of the representation $(\bsL,\bsf)$, as claimed.  We chose in Section~\ref{sec:quiver.main.proof} above to present an argument using Proposition~\ref{pro:bdop}, to illustrate the interplay between the model-theoretic results we invoke and the algebraic and combinatorial arguments.  Model-theoretic claims may often be treated as ``black boxes'' that can be swapped into existing proofs to obtain variants, refinements, or generalizations of their conclusions.

In the proof of Theorem~\ref{thm:quiver.local}, by contrast, as far as we are aware we really need the power of the results of~\cite{HMR/18, Nguyen/19} to treat cases where $A$ has positive characteristic.  
\end{rem}

\enlargethispage{-3\baselineskip}
\begin{acknowledgements}
We are very grateful to Itay Glazer, Immanuel Halupczok, Seungjai Lee, and Christopher Voll for helpful correspondence and comments on earlier versions of this work and to the anonymous referees for suggestions that have improved the exposition.
\end{acknowledgements}

\def\cprime{$'$}
\providecommand{\bysame}{\leavevmode\hbox to3em{\hrulefill}\thinspace}
\providecommand{\MR}{\relax\ifhmode\unskip\space\fi MR }
\providecommand{\MRhref}[2]{%
  \href{http://www.ams.org/mathscinet-getitem?mr=#1}{#2}
}
\providecommand{\href}[2]{#2}

\end{document}